\documentclass{article}
\usepackage{amsthm,amssymb,amsmath,enumerate,graphicx,epsf}
\usepackage[dvips, bookmarks, colorlinks=false]{hyperref}

\newcommand{\NOTESON}{0}
\newcommand{\Debug}{0}
\newcommand{\COLORON}{0}

\usepackage[usenames]{color} 

\newcommand{\comment}[1]{}
\newcommand{\COMMENT}[1]{}

\newcommand{\defi}[1]{{\color{Green}\emph{#1\/}}}

\newtheorem{proposition}{Proposition}[section]

\newtheorem{theorem}[proposition]{Theorem}
\newtheorem{corollary}[proposition]{Corollary}

\newtheorem{conjecture}{{\color{red}Conjecture}}[section]

\newtheorem{problem}[conjecture]{{\color{red}Problem}}

\newtheorem{examp}{Example}[section]

\newcommand{\FIG}{0}

\ifnum \NOTESON = 1 \newcommand{\note}[1]{ 

	\ 

	{\color{Blue} NOTE: \color{Turquoise}{\small  \tt \begin{minipage}[c]{0.8\textwidth}  #1 \end{minipage} \ignorespacesafterend }} 
	
	\ 
	
	}
\else \newcommand{\note}[1]{} \fi

\ifnum \Debug = 1 
\else  \fi

\ifnum \FIG = 1 \newcommand{\fig}[1]{Figure ``{#1}''}
\else \newcommand{\fig}[1]{Figure~\ref{#1}} \fi

\ifnum \Debug = 1 \usepackage[notref,notcite]{showkeys}
\fi

\ifnum \COLORON = 0 \renewcommand{\color}[1]{}
\fi

\newcommand{\showFigR}[3]{
   \begin{figure}[!hbt]
   \centering
   \noindent
   \includegraphics[width=#3]{#1}
   \caption{\small #2}
   \label{#1}
   \end{figure}
}

\newcommand{\showFig}[2]{
   \begin{figure}[htbp]
   \centering
   \noindent
   \epsfbox{#1.eps}
   \caption{\small #2}
   \label{#1}
   \end{figure}
}

\newcommand{\N}{\mathbb N}

\newcommand{\sm}{\backslash}

\newcommand{\nin}{\ensuremath{{n\in\N}}}

\newcommand{\pths}[2]{\ensuremath{#1}\text{--}\ensuremath{#2}~paths}

\newcommand{\Tr}[1]{Theorem~\ref{#1}}
\newcommand{\Sr}[1]{Section~\ref{#1}}

\newcommand{\obda}{without loss of generality}

\newcommand{\fe}{for every}

\newcommand{\ti}{there is}

\newcommand{\labtequ}[2]{ \begin{equation} \label{#1} 	\begin{minipage}[c]{0.9\textwidth}  #2 \end{minipage} \ignorespacesafterend \end{equation} } 
\newcommand{\mySection}[2]{}
\title{An Eberhard-like theorem for pentagons and heptagons}

\author{Matt DeVos \\
  {Department of Mathematics}\\
  {Simon Fraser University}\\
  {Burnaby, B.C. V5A 1S6} \\
  {\tt mdevos@sfu.ca}
\and
  Agelos Georgakopoulos\thanks{Supported in part by a GIF grant and in part by an FWF grant.}~\thanks{This work was conceived during a visit of the second author at the Simon Fraser University; we thank the SFU for supporting this visit.}\\
  {Technische Universit\"at Graz}\\
  {Steyrergasse 30, 8010}\\
  {Graz, Austria}\\
  {\tt georgakopoulos@tugraz.at}
\and
  Bojan Mohar\thanks{Supported in part by the
  Research Grant P1--0297 of ARRS, by an NSERC Discovery Grant
  and by the Canada Research Chairs program.}~\thanks{On leave from:
  IMFM \& FMF, Department of Mathematics, University of Ljubljana, Ljubljana,
  Slovenia.}\\
  {Department of Mathematics}\\
  {Simon Fraser University}\\
  {Burnaby, B.C. V5A 1S6} \\
  {\tt mohar@sfu.ca}
\and
   Robert \v{S}\'{a}mal\thanks{Supported in part by PIMS postdoctoral fellowship 
   at the Simon Fraser University, by Institute for Theoretical Computer Science 
   (ITI) and by Department of Applied Mathematics.
   Partially supported by grant GA \v{C}R P201/10/P337.}\\
   KAM \& ITI\\
   Charles University\\
   Prague, Czech Republic\\
   {\tt samal@kam.mff.cuni.cz}
}

\date{}
\begin{document}
\maketitle

\begin{abstract}
Eberhard proved that for every sequence $(p_k), 3\le k\le r, k\ne 6$
of non-negative integers satisfying Euler's formula $\sum_{k\ge3} (6-k) p_k = 12$,
there are infinitely many values $p_6$ such that there exists a simple 
convex polyhedron having precisely $p_k$ faces of size $k$ for every 
$k\ge3$, where $p_k=0$ if $k>r$.
In this paper we prove a similar statement when non-negative integers
$p_k$ are given for $3\le k\le r$, except for $k=5$ and $k=7$ (but including $p_6$). We prove that 
there are infinitely many values $p_5,p_7$ such that there
exists a simple convex polyhedron having precisely $p_k$ faces of size 
$k$ for every $k\ge3$.
%, where $p_k=0$ if $k>r$. 
We derive an extension to arbitrary closed surfaces, yielding maps of arbitrarily high face-width. Our proof suggests a general method for obtaining 
results of this kind.
\end{abstract}

\section{Introduction}

Consider a cubic (i.e., 3-regular) plane graph, and let $p_k$ ($k\ge 1$) denote the number of its
$k$-gonal faces. It is a simple corollary of Euler's formula that
\begin{equation}
  \sum_{k\ge1} (6-k) p_k = 12.
\label{eq:basic}
\end{equation}
It is natural to ask for which sequences $(p_k)_{k\ge1}$ satisfying (\ref{eq:basic})
there exists a cubic plane graph whose face sizes comply with the sequence $(p_k)$.
This question is even more interesting when additional restrictions on the graph are
given. The most important case is to consider graphs of 3-dimensional convex polyhedra,
so called {\em polyhedral graphs\/}.
By Steinitz's Theorem, this is the same as requiring the graphs to be 3-connected. An important subcase is when the polyhedra are \defi{simple}, in other words, when the corresponding graphs are cubic.

The general problem about the existence of polyhedral graphs with given face sizes is still wide open. 
However, there are many special cases that have been solved. 
For example \cite[Theorem 13.4.1]{Gru1}, it is known that there exists a simple
polyhedron with six quadrangular faces and $p_6$ faces of size six if and only
if $p_6\ne1$; and there exists a simple polyhedron with twelve pentagonal faces 
and $p_6$ faces of size six (a ``fullerene'' graph) if and only if $p_6\ne1$.
A similar case of four triangular faces and $p_6$ faces of size 6 has infinitely
many exceptions: such a polyhedron exists if and only if $p_6$ is even.
We refer to \cite{Gru1} 
for a complete overview. The most fundamental result in this area 
is the following classical theorem of Eberhard \cite{Eb}, stating that there is always 
a solution provided we are allowed to replace $p_6$ (whose value does not affect the satisfaction of~\eqref{eq:basic}) by a large enough integer.

\begin{theorem}[Eberhard \cite{Eb}]
\label{thm:Eberhard}
For every sequence $(p_k), 3\le k\le r, k\ne 6$ of non-negative integers
satisfying\/ $(\ref{eq:basic})$, there are infinitely many values $p_6$ such that there
exists a simple convex polyhedron having precisely $p_k$ faces of size $k$ for every 
$k\ge3$, where $p_k=0$ if $k>r$.
\end{theorem}

Eberhard's proof is not only long and messy but also some of its parts may
not satisfy today's standards of rigor. Gr\"unbaum \cite{Gru1} gave a simpler complete
proof utilizing graphs and Steinitz's Theorem. This result was strengthened by 
Fisher \cite{Fi1} who proved that there is always a value of $p_6$ that satisfies
$p_6 \le p_3+p_4+p_5+\sum_{k\ge7} p_k$.

Gr\"unbaum also considered a 4-valent analogue of Eberhard's theorem.
Fisher \cite{Fi2} proved a similar result for 5-valent polyhedra, establishing
existence for all sequences of face sizes with $p_4\ge 6$ that comply with Euler's formula. 

Various other generalizations of Eberhard's theorem have been discovered.
Papers by Jendrol' \cite{Je93a,Je93b} give a good overview and bring some of
today's most general results in this area. Some other relevant works include
\cite{Ba,BGH,En,Gru2,JJ2}.
Several papers treat extensions of Eberhard's theorem to 
the torus \cite{Gri,JJ1,Za1,Za2} and more general surfaces \cite{Je93a}.
It is worth pointing out that on the torus there is precisely one
admissible sequence (namely $p_5=p_7=1$ and $p_i=0$ for $i\notin\{5,7\}$),
for which an Eberhard-type result with added hexagons does not hold \cite{JJ1}.

%%%%%%%%%%%%%%%%%%%%%%%%%%%%%%%%%%%%%%%%%%%%%
% End of overview, description of our results
%%%%%%%%%%%%%%%%%%%%%%%%%%%%%%%%%%%%%%%%%%%%%

In this paper we consider a similar problem that is also motivated by 
(\ref{eq:basic}). Let us suppose that we are given face sizes as before
but we are only allowed to change $p_5$ and $p_7$ (or $p_{6-t}$ and $p_{6+t}$
for some $t$, $1\leq t\leq 3$). In this case, we think 
of $p_k$ (for $k\ge 3$, $k \ne 5, 7$) as being fixed and $p_5,p_7$ as being free 
to choose. Equation (\ref{eq:basic}) determines
the difference $s = p_7 - p_5$, and we are asking if there exist $p_5$
and $p_7=p_5+s$ with a polyhedral realization. %Our main result, Theorem~\ref{ebe57}, 
We give an affirmative answer to this question, and derive an extension solving the corresponding problem on an 
arbitrary closed surface. Our construction gives simple (i.e., 3-regular) polyhedral maps on a surface, and one can impose the additional conditions that these maps
have large face-width and their graphs be 3-connected. More precisely, we prove

\begin{theorem}\label{thintro}
Let  $(p_k), 3\le k\le r, k\ne 5,7$ be a sequence of non-negative integers,
let $S$ be a closed surface, and let $w$ be a positive integer. 
Then there exist infinitely many pairs of integers $p_5$ and $p_7$
such that \ti\ a 3-connected cubic  map realizing $S$, with face-width at least $w$, having precisely $p_k$ faces of size $k$ \fe\ $k \in \{3, \ldots, r\}$.
\end{theorem}

It is worth observing that we also fix the number $p_6$ of hexagonal faces. Secondly, observe that the extension of Eberhard's Theorem to a surface $S$ other than the sphere needs an adjustment in (\ref{eq:basic}); the right hand side has to be replaced
by $6\chi(S)$ where $\chi(S)$ is the Euler characteristic of $S$. However, in our setting the formula adjusts itself
by using an appropriate number of pentagons and heptagons.

Finally, as we point out in \Sr{other}, our proof suggests a general method for obtaining results of this kind.

\section{Definitions}

A finite sequence $p=(p_3,p_4,\ldots, p_r)$ is \defi{plausible} for a closed surface $S$ if 
\begin{equation} \label{plau}
  \sum_{3\leq k \leq r} (6-k)p_k= 6\chi(S)
\end{equation}
where $\chi(S)$ is the Euler characteristic of $S$. By Euler's formula,
\eqref{plau} is a necessary condition for the existence of a cubic graph
embeddable in $S$ with precisely $p_k$ $k$-gons for $3\leq k \leq r$ and no other faces. If there exists a cubic graph which is 2-cell embeddable in $S$ with
precisely $p_k$ faces of size $k$ for $3\leq k \leq r$ 
and no other faces, then we say that $p$ is \defi{realizable} in $S$. 
If\/ $\sum_{3\leq k \leq r} (6-k)p_k= 0$, then we call 
$p$ a \defi{neutral} sequence. For any two such sequences, one can consider their
sum which is defined in the obvious way. Let us observe that the sum of 
a neutral sequence and a plausible sequence is a plausible sequence. We would like to understand in this context which plausible sequences are realizable, and try to do so by asking when a sum of a plausible sequence with an appropriate  neutral sequence is realizable. For the neutral sequence $(0,0,0,1)$ this is Eberhard's theorem.

The most important building block in both Eberhard's as well as our proofs is a construction called a triarc. A \defi{triarc} is a plane graph $T$
such that the boundary $C$ of the outer face of $T$ is a cycle, and moreover the following conditions are satisfied (examples are the graphs in 
\fig{443224tr} with the half-edges in the outer face removed):
\begin{itemize}
\item every vertex of $T - C$ has degree 3 in $T$;
\item $C$ contains distinct vertices $x,y,z$ of degree 2 (called the \defi{corners} of the triarc) such that the degrees (in $T$)
of the vertices on each of the three paths in $C - \{x,y,z\}$ alternate between 2 and 3, starting and ending with a vertex of degree 2.
\end{itemize}
A \defi{side} of a triarc $T$ as above is a subpath of $C$ that starts and ends at distinct corners of $T$ and does not contain the third corner. The
\defi{length} of a side $P$ of $T$ is the number of inner vertices of degree 2 on $P$; note that although the corners of a triarc have degree 2, they
are not counted when calculating the lengths of its sides. 
A triarc with sides of lengths $a$, $b$, $c$ is called an 
\defi{$(a,b,c)$-triarc}. Of course, we can flip or rotate such a triarc 
and consider it, for example, as a $(b,a,c)$-triarc.

%The definition of a triarc is motivated by a triangular piece of a hexagonal 
%grid, cf.~Figure~\ref{glueing57}, which contains an $(8,8,8)$-triarc
%in the lower left corner. 
%(However, the definition allows for more general construction, see the example on
%Figure~\ref{443224tr} (right) that will be one of our basic building blocks.)
%As a consequence, 
Triarcs are very versatile tools. Firstly, if the length of some side of a triarc $T$ equals the length of some side of another triarc
$R$, then $T$ and $R$ can be glued together along those sides to yield a new plane graph with all inner vertices having degree 3; see for example
\fig{fig0}. Secondly, every triarc $T$ has zero total curvature; to see this, take two copies of $T$, turn one of
them upside down, glue them along a common side to obtain a `parallelogram' 
(see \fig{fig0} again), and identify opposite sides of this parallelogram
to obtain a graph embeddable in the torus.
But perhaps the most important property of triarcs is the possibility to `glue' them together to obtain larger triarcs; we describe this operation below. 

\showFigR{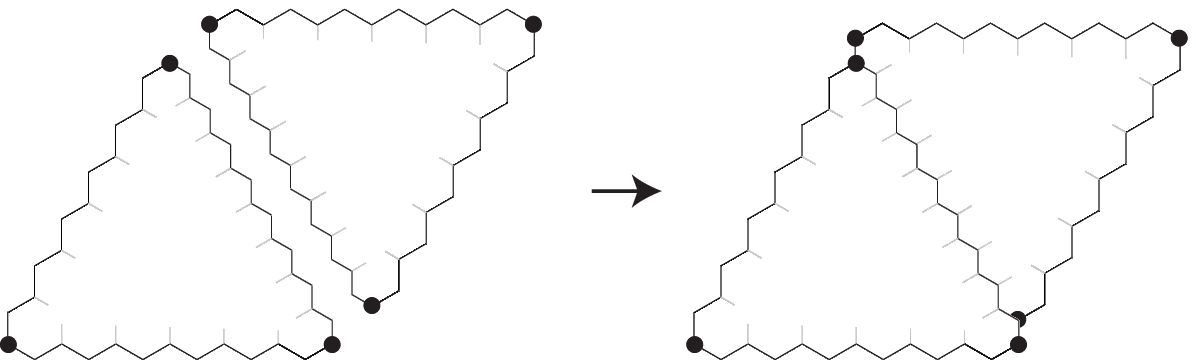}{Glueing two triarcs along sides of equal length. The dots represent the corners of the triarcs.}{10.5cm} 

\showFigR{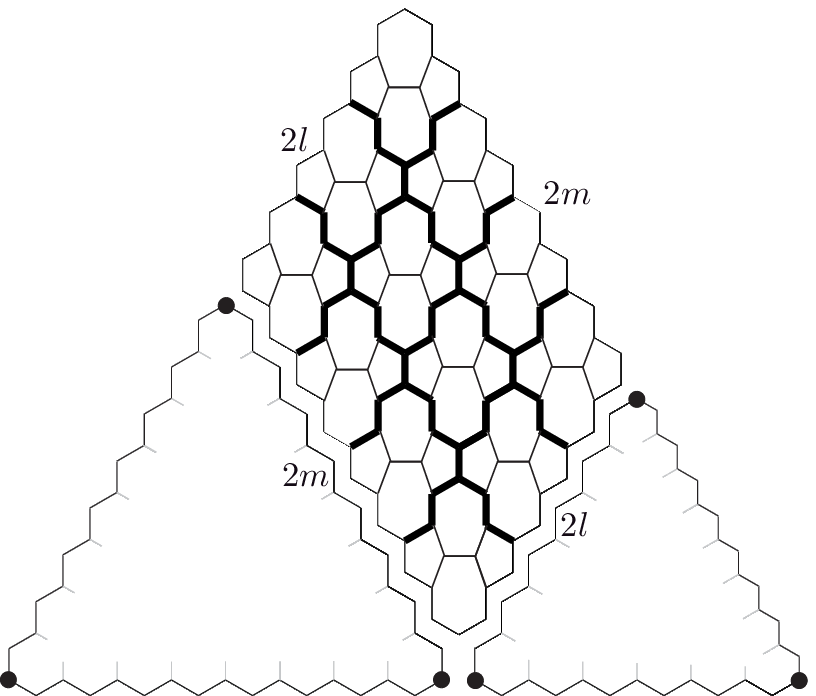}{Glueing two triarcs with two sides of even length together using the tile of \fig{2x2tile}.}{7cm}

\showFigR{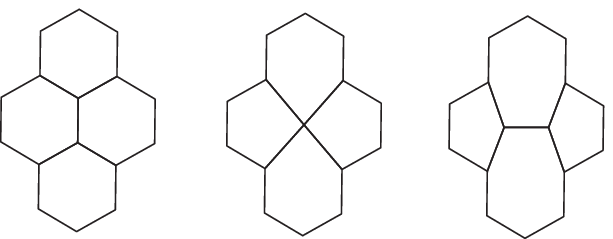}{In a configuration of 4 hexagons we may contract the central edge and then ``uncontract'' it in the other direction. A ``tile'' consisting of two 
pentagons and two heptagons results; we use such tiles in \fig{glueing57}.}{4cm}

\newcommand{\htile}{$(5,7)$-tile}
\renewcommand{\htile}{H-tile}
Suppose we have an $(a_1, b_1, c_1)$-triarc and an $(a_2, b_2, c_2)$-triarc such that  $b_1 = 2m$ and $c_2=2l$ are even. Then, we may combine these triarcs (and several pentagons and heptagons) to construct an
$(a_1+a_2, b_1+b_2, c_1+c_2)$-triarc.
To do this, we identify a corner (and an incident edge) of the first triarc with a corner (and an edge) of the second triarc ---see Figure~\ref{glueing57}--- so that the two identified corners yield a vertex of degree~3 on a side 
of length $a_1+a_2$ in a new triarc.
Then, we can add a \defi{``parallelogram''} consisting of hexagons to obtain an $(a_1+a_2, b_1+b_2, c_1+c_2)$-triarc.
However, we do not want to add hexagons. Instead, we decompose the parallelogram into tiles each consisting of four hexagons as depicted in Figure~\ref{glueing57}, and replace each
of these tiles by two pentagons and two heptagons as indicated in Figure~\ref{2x2tile}. The ``tile'' on the right of Figure~\ref{2x2tile} will be used several times in the sequel, and we shall refer to it as a \defi{\htile}.

We are going to use this operation of glueing two triarcs into a larger one several times in the following section.

\section{Proof of \Tr{thintro}} 

\label{main}

We are ready to state and prove our main result. 
Let us observe that, unlike Eberhard's Theorem, we do not need to assume 
that the given face-sizes form a plausible sequence (although we make this
assumption in the formulation of the theorem) because given a sequence 
 $(p_k), 3\le k\le r, k\ne 5,7$, the sequence can always be appended by 
appropriate values $p_5$ and $p_7$ to become plausible. 

\begin{theorem}\label{ebe57}
Let $p=(p_3,p_4,\ldots,p_r)$ be a plausible sequence for the sphere. Then there 
exist infinitely many integers $n\in \N$ such that the sequence\\ $p+n\cdot (0,0,1,0,1)$ 
is realizable in the sphere.
\end{theorem} 
% *** ---- ***
\begin{proof}
We will give an explicit construction of a cubic graph embeddable in the sphere whose face sequence is of the form $p+n\cdot (0,0,1,0,1)$.
The rough plan for this is as follows. For each face imposed by the sequence $p$,
we create a \defi{basic triarc} containing this face as well as some pentagons and heptagons. Then,
we glue all these triarcs together and extend to a triarc with sides of suitable lengths. Finally, we construct a new triarc having the same side lengths, and glue these
two triarcs together (as explained later) to obtain the desired graph embedded in the sphere.

To construct a basic triarc for a $k$-gon (we will make $p_k$ copies of it), we surround the $k$-gon by three heptagons and $k-3$ pentagons
as shown in the right half of \fig{443224tr} (where the $k$-gon we are surrounding happens to be a pentagon). Note that we can always make the basic triarc isosceles with the equal sides having even length. We call the $k$-gon we started with the \defi{nucleus} of this triarc.

%\showFig{basic}{A basic triarc \sss.}

Having constructed all basic triarcs, our next step is to glue them all together to obtain a single triarc $T$ containing them all. We do so
recursively, attaching one basic triarc at a time as shown in \fig{glueing57}, where we use many copies of the \htile\ in order to
build the parallelogram needed. Each time we use this glueing operation we are assuming that both triarcs in \fig{glueing57} are isosceles, with the equal sides having even length, and align them so that the two equal even sides are the upper left and upper right side. Note that the resulting triarc is also isosceles with two equal sides of even length. Thus, we can continue recursively to glue all basic triarcs into one isosceles triarc $T$.

%We are assuming here that both triarcs in \fig{glueing57} are isosceles, with the equal sides having even length, but
%note that the resulting triarc will also have that property. Thus, since the basic triarcs are by construction isosceles with equal sides of even
%length, we can indeed perform this operation recursively to glue them all into one isosceles triarc $T$.
%(note however, that in order to be able to
%perform the operation of \fig{glueing57} it suffices if the sides of the triarcs along which we are glueing have even length; we only insist that the
%triarcs be isosceles because we want the final triarc $T$ to be isosceles).

Our next aim is to enlarge $T$ into an equilateral triarc $T'$ with sides of length $n$, where $n$ is a multiple of 8 and satisfies 
$n\equiv 2 \pmod 3$, using only pentagons and heptagons. To this end, we will use the glueing operation of \fig{glueing57} and many copies of a $(4,4,3)$-triarc and a $(2,2,4)$-triarc.
\fig{443224tr} shows how to construct those triarcs with pentagons and heptagons only. 

%\showFig{443tr}{A $(4,4,3)$-triarc.}
\showFigR{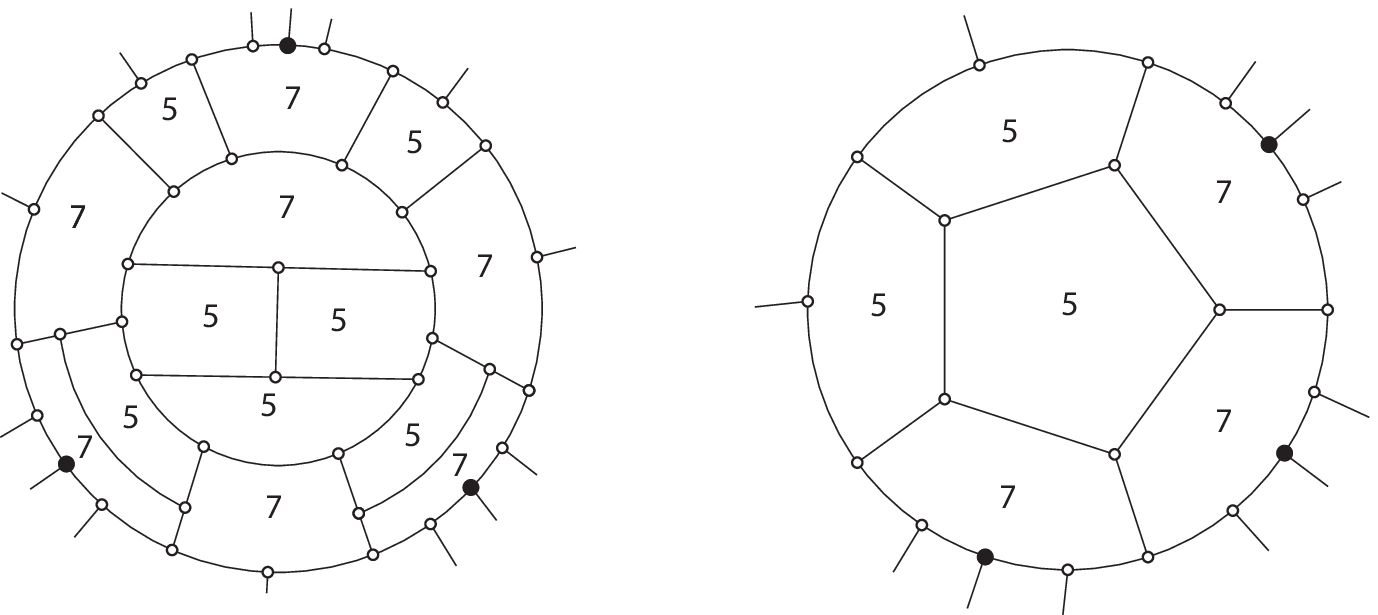}{A $(4,4,3)$-triarc and a $(2,2,4)$-triarc.}{10cm}
%\showFig{224tr}{A $(2,2,4)$-triarc.}

Note that glueing $T'$ with a $(4,4,3)$-triarc (as in \fig{glueing57}) keeps it isosceles and decreases the difference of lengths between the ``base''
and the other two sides by 1, while glueing with a $(2,2,4)$-triarc increases that difference by 2. Thus, recursively glueing with such triarcs we can
enlarge $T$ into an equilateral triarc $S$ with sides of even length. 

Moreover, using the glueing operation of \fig{glueing57} three times, once with a $(2,2,4)$-triarc and twice with a $(4,4,3)$-triarc, %as shown in \fig{plus10}, 
we can increase the side-lengths by $(2,2,4)+(4,4,3)+(4,4,3)=(10,10,10)$. Thus we can increase the length of each side of $S$ by 10. Since $10 \equiv 1
\pmod 3$, we can use this operation to enlarge $S$ into an equilateral triarc $S'$ with even sides of length $2 \pmod 3$. Moreover, since performing
this operation three times increases the length of each side by 30, and $30 \equiv 6 \pmod 8$, we can enlarge $S'$ into an equilateral triarc $T'$ with the length
of each side being a multiple of $8$ and congruent to $2$ modulo $3$.  

%\showFig{plus10}{Increasing the length of each side of a triarc by 10.}

Next, we are going to construct a triarc $R$ that has the same 
side lengths as $T'$ but consists of pentagons and heptagons only. 
%As shown in \fig{888tr}, it is possible to glue three $(2,2,4)$-triarcs 
%together to obtain an equilateral triarc $D$ with sides of length 8. 
By glueing together a $(2,2,4)$-triarc, a $(2,4,2)$-triarc and a $(4,2,2)$-triarc (that is, the same triarc in three different
rotations) we get an $(8,8,8)$-triarc, which we will call~$D$.
Since the sides of $T'$ have length a multiple of 8, by glueing copies of $D$ together recursively as in \fig{glueing57} we can indeed construct a 
triarc~$R$ that has the same dimensions as $T'$.

%\showFig{888tr}{Glueing three $(2,2,4)$-triarcs together (as in \fig{glueing57}) to obtain an equilateral triarc with sides of length 8.}

We can now combine $R$ and $T'$ together to produce a cubic graph tiling the sphere as shown in \fig{cyl57}. By construction, this graph has for every 
$k\in \N \sm \{0,1,2,5,7\}$ precisely $p_k$ faces of size $k$, and moreover it has at least $p_5$ pentagons and at least $p_7$ heptagons. Thus its face sequence is of the form $p+(0,0,n,0,m)$ for some $n,m\in \N_+$. Since both $p$ and $p+(0,0,n,0,m)$ satisfy Euler's formula (the former by assumption, the latter because the plane graph we just constructed implements it), we have $n=m$. 

This completes the construction and shows the existence of one particular value of $n$ as desired. However, observe that the construction of $T'$ and $R$ allows us to make
the side lengths of these triarcs arbitrarily large. This shows that we can get
examples for infinitely many values of $n$ and thus completes the proof. 
\end{proof}

\showFig{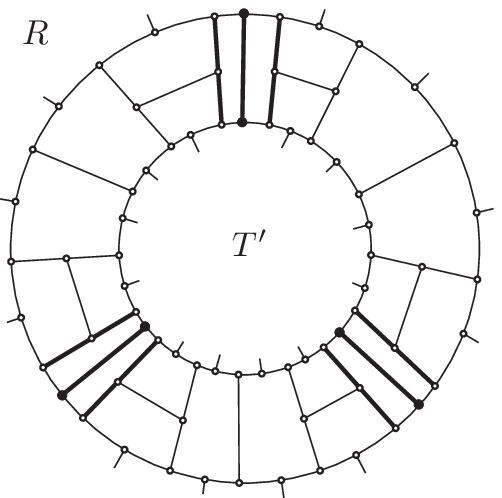}{Glueing $R$ and $T'$ together along a ``ring'' consisting of pentagons and heptagons. This operation is possible because we made sure that every side 
of $T'$, and thus also of $R$, has length congruent to $2 \pmod 3$.}%{10cm}

We now turn from planar graphs to maps on arbitrary (compact) surfaces.
A \defi{map} on a surface $S$ is a graph together with a 2-cell embedding in $S$.
A map is \defi{polyhedral} if all faces are closed disks in the surface and 
the intersection of any two faces is either empty, a common vertex or a common
edge. If the graph of the map is cubic, then we say that the map is \defi{simple}.

A cycle contained in the graph of a map is \defi{contractible} if it bounds a disk
on the surface. The \defi{edge-width} of a map $M$ is the length of a shortest 
non-contractible cycle in $M$. The \defi{face-width} of $M$ is the minimum number
of faces, the union of whose boundaries contains a non-contractible cycle.
We refer the reader to \cite{MT} for more about the basic properties and the 
importance of these parameters of maps. At this point we only note that a map
is polyhedral if and only if its graph is 3-connected and its face-width
is at least three, see \cite[Proposition~5.5.12]{MT}. We also note that
if $r$ is the largest size of a face of $M$, then the edge-width of $M$
cannot exceed $\tfrac{r}{2}$ times the face-width of $M$.

%Our next result implies that we can replace the sphere by any closed surface in \Tr{ebe57}.
We now restate and prove our main result, \Tr{thintro}.

\begin{corollary}\label{cor}
Let  $(p_k), 3\le k\le r, k\ne 5,7$ be a sequence of non-negative integers,
let $S$ be a closed surface, and let $w$ be a positive integer. 
Then there exist infinitely many pairs of integers $p_5$ and $p_7$
such that the sequence $(p_3,p_4,p_5,\dots,p_r)$ is realizable in $S$ and \ti\ a 3-connected realizing cubic map of face-width at least $w$.
\end{corollary} 
% *** ---- *** 
\begin{proof}
Let us first describe a construction that does not necessarily achieve the desired face-width; we will later explain how to modify this 
construction in order to get large face-width. 

The rough sketch of this construction is as follows.
Firstly, we increase the number of hexagons in the sequence $(p_k)$ to $p'_6:= p_6 + 2h+c$, where $h$ is the number of handles of $S$ and $c$ the number of its crosscaps (by the surface classification theorem we may assume that one of $h,c$ is zero, but we do not have to). It follows from \Tr{ebe57} that we can increase the numbers $p_5$ and $p_7$ of this sequence to some appropriate values so that the resulting sequence $p'$ is realized by a map on the sphere. We will then use the $2h+c$ auxilliary hexagons of this map we added above to introduce some handles and/or crosscaps. After doing so, all auxilliary hexagons will have disappeared, and we obtain a map on $S$ whose sequence of faces differs from $(p_k)$ by some pentagons and heptagons only.

%The given sequence can be extended 
%by choosing appropriate values $p_5,p_7$ so that $p=(p_3,p_4,p_5,\dots,p_r)$
%becomes plausible for the surface $S$. Moreover, we extend $p$ further by adding some auxiliary hexagons; more precisely, let $p':=p+ (0,0,0,k)$ for $k:=2h+c$ where $h$ is the number of handles of $S$ and $c$ the number of its crosscaps (by the surface classification theorem we may assume that one of $h,c$ is zero, but we do not have to). 

More precisely, similarly to the proof of \Tr{ebe57}, we construct a basic triarc for each face in $p'$, but with one modification: for each hexagon we construct a
triarc like the one in \fig{newhex} (on the left) rather than one with two even sides of equal lengths (in fact, we need this modification for the auxiliary hexagons only, but we might as well use it for the original hexagons in $p$ as well). 

\begin{figure}[ht]
   \centering
   \noindent
   \includegraphics[width=3.8cm]{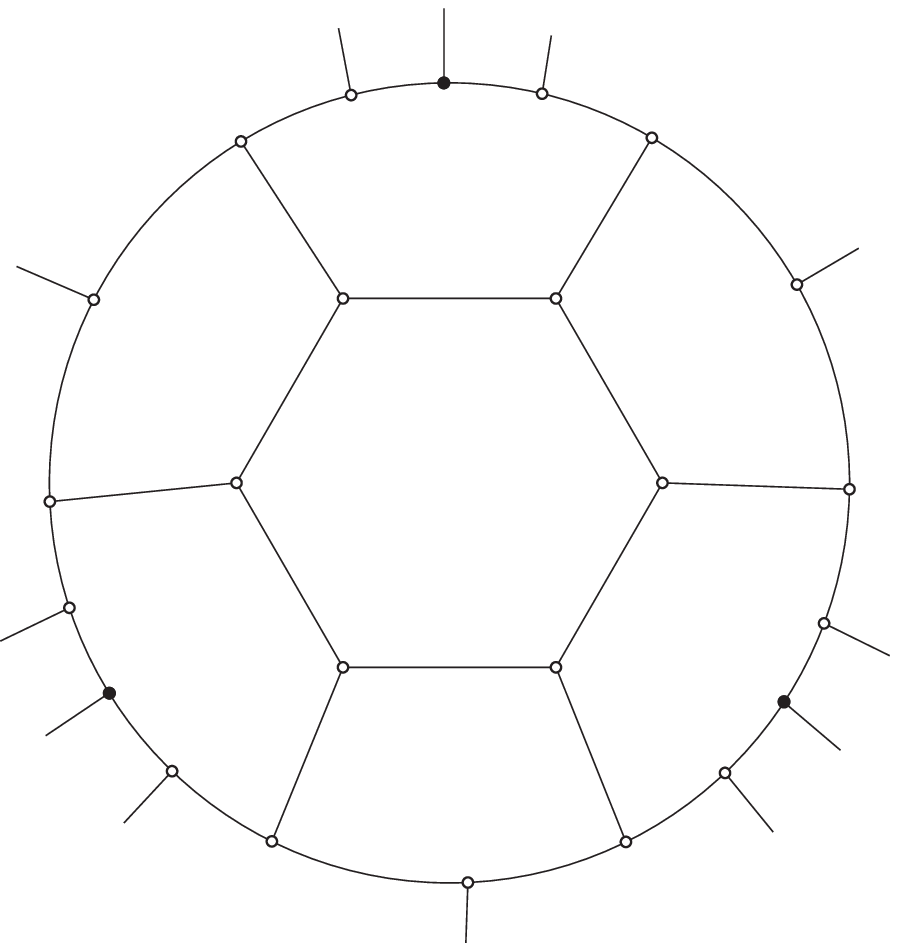}
   \qquad\qquad
   \includegraphics[width=4cm]{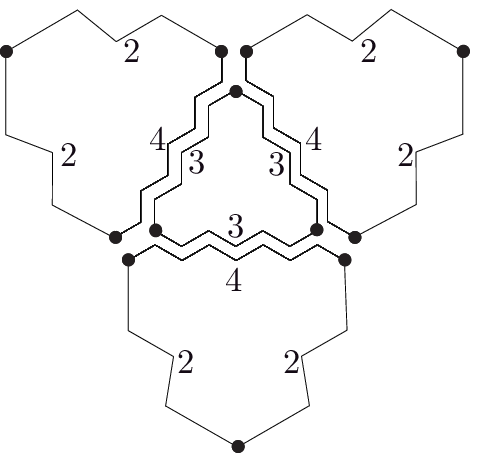}
   \caption{\small On the left: the new basic triarc for a hexagon.
    On the right: extending the triarc from the left into an equilateral 
    triarc with even sides.
   }
   \label{newhex}
\end{figure}

Next, we proceed as in \Tr{ebe57} to glue all basic triarcs together into one triarc $T$. However, since we now have basic triarcs with all sides odd
(the ones of \fig{newhex}), the glueing operation of \fig{glueing57} will not work for these triarcs. For this reason, we first extend each such triarc
into an equilateral triarc with even sides using three copies of the $(2,2,4)$-triarc of \fig{443224tr} as shown in \fig{newhex} (right).

\showFigR{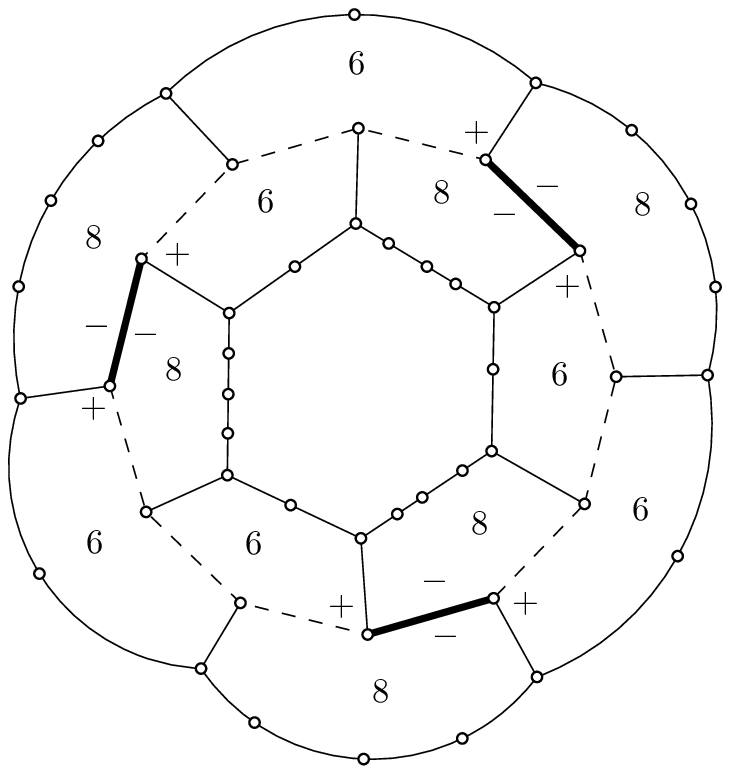}{The situation arising after introducing a handle. The 12-cycle $C$ consists of the dashed and the thick edges.}{5cm}

We continue imitating the proof of \Tr{ebe57} to obtain a cubic graph $G$ embedded in a homeomorphic copy $S'$ of the sphere that contains all basic triarcs. We will now perform some cutting and glueing operations on both $S'$ and $G$ to obtain a new surface, homeomorphic to $S$, with a cubic graph $G'$ embedded in it.
%Note that $p'$ has the same curvature as $p$. Since the sphere has the highest Euler characteristic among all closed surfaces, there is a positive integer $n$ such that the sequence $s:= p' + (0,0,n)$ is plausible for the sphere. 

Suppose that $h>0$. Then, pick $h$ pairs $(F_1,F'_1), \ldots, (F_h,F'_h)$ of hexagonal faces of $G$, such that all the faces $F_i$ and $F'_i$ are distinct (there are enough hexagonal faces by our choice of the sequence $p'$). Now for each pair $(F_i,F'_i)$ perform
the following operations. Cut out the two discs of $S'$ corresponding to $F_i,F'_i$, and glue their boundaries together with a half-twist; that is, each
vertex of the boundary of $F_i$ is identified with the midpoint of an edge of $F'_i$ and vice-versa. This operation creates a handle in $S'$, and the
embedded graph remains cubic; however, it also gives rise to some unwanted faces: the size of each face that was incident to $F_i$ or $F'_i$ has now
been increased by 1. We thus have the situation depicted in \fig{handle}, where $C$ is the cycle of length 12 resulting from the boundaries of $F_i$
and $F'_i$. Recall that since every hexagon is put in a basic triarc like the one in \fig{newhex}, the sizes of the faces on each side of $C$
alternate between 6 and 8 as shown in \fig{handle}. But now, contracting and uncontracting each of the three thick edges
(in the way explained in \fig{2x2tile}) turns each of the faces incident with $C$ into a heptagon. 

\showFigR{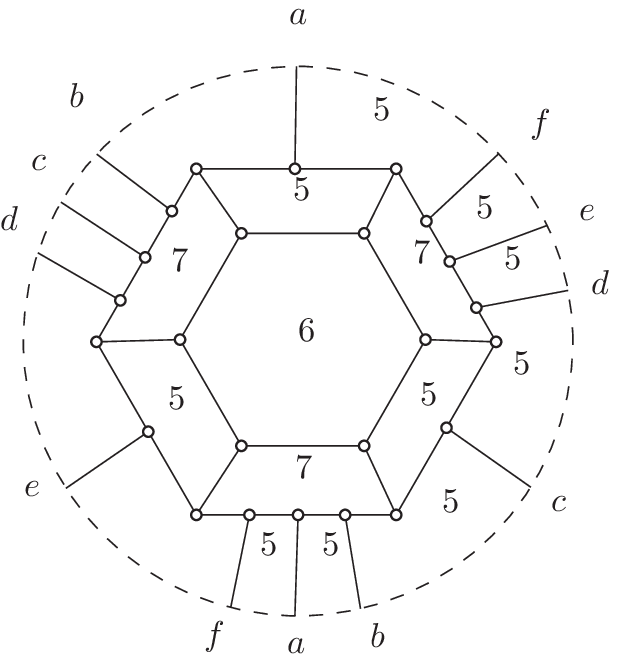}{The gadget used to create a crosscap inside a hexagon.}{6cm}

On the other hand, if $c>0$, then pick $c$ distinct hexagonal faces $F_1, \ldots, F_c$, and for every $i$ cut out the disc corresponding to $F_i$ and
glue in its place the outside of the hexagon of \fig{crosscap} with a half twist. Each such operation gives rise to a new crosscap, but also to unwanted faces
just like in \fig{handle}. But again, contracting and uncontracting each of the three thick edges we can turn all these unwanted hexagons and octagons into heptagons.

Thus, after all these operations have been completed, we obtain a surface with $h$ handles and $c$ crosscaps with a cubic graph embedded in it whose
face sequence is $p + (0,0,n,0,m)$ for some $n,m\in \N_+$. Note that all 
auxiliary hexagons in $p' - p$ have disappeared after the above operations. 
%As in the previous proof, $n=m$ must hold by Euler's formula.---NO; WE ARE NOT ASSUMING PLAUSIBILITY HERE

It is easy to check that our maps are 3-connected. 
Indeed, our ``building blocks'' -- the basic triarcs and the triarcs of Figure~\ref{443224tr} -- are 
3-connected after suppressing the degree~2 vertices. This property is also true
for triarcs in Figure~\ref{newhex}, and it is maintained by the glueing operation of Figure~\ref{glueing57}. 
By glueing two triarcs along a ``circumference'', using the ring in Figure~\ref{cyl57}, we get a 3-connected graph.
The gadget we used for introducing cross-caps (Figure~\ref{crosscap}) is 3-connected, and it is not hard to see 
that we maintain 3-connectivity when adding this gadget or when creating a handle as depicted in Figure~\ref{handle}. 

\medskip
It remains to discuss how to modify this construction to obtain maps with arbitrarily large face-width.
By the remark preceding Corollary \ref{cor}, it suffices to construct maps with arbitrarily
large edge-width $z$ since the face sizes are bounded from above by $r$.
This is achieved as follows. 

First of all, we make every basic triarc used in 
the construction large enough that the distance from its nucleus to the boundary of the triarc is at least $z$ and the length of each side of each triarc is at least $3z$. This can be achieved by the method we used in the proof of \Tr{ebe57} to enlarge $T$ into an equilateral triarc $T'$.
to the boundary of the triarc is at least $z$ and the length of each side of each triarc is at least $3z$. 
To achieve this, we first glue the triarc with several $(2,2,4)$-triarcs (or any other triarcs) 
both on the left and on the right, to obtain a large triarc with the nucleus in the middle of the bottom side.
Then we possibly glue it with a $(4,4,3)$-triarc to create a triarc with all sides even.
Finally, we rotate the triarc by $120^\circ$ and perform more glueing with $(2,2,4)$-triarcs to get the nucleus
away from the boundary. (The notions ``left'', ``right'', and ``$120^\circ$'' in this paragraph refer to the glueing operation 
of Fig.~\ref{glueing57}.)

%To achieve this, we first glue the triarc with several $(2,2,4)$-triarcs (or any other triarcs)  both on the left and on the right, to obtain a large triarc with the nucleus in the middle of the bottom side. Then we possibly glue it with a $(4,4,3)$-triarc to create a triarc with all sides even. Finally, we rotate the triarc by $120^\circ$ and perform more glueing with $(2,2,4)$-triarcs to get the nucleus away from the boundary. 

Next, we replace the auxiliary 
hexagons used in order to add handles and crosscaps with $6N$-gons,
where $N$ is odd and greater than $z/2$. Of course, this will force us to add some more pentagons to our sequence $p_k$ to make it plausible. Note that we can generalize the triarc on the left of \fig{newhex} %even if $S$ has only crosscaps!!
so that  the inner 6-gon is replaced by a $6N$-gon 
surrounded by three heptagons and $6N-3$ pentagons, arranged in a symmetric way so that any two heptagons separate $2N-1$ pentagons from the rest.  We will make use of the fact that $2N-1$ is odd.  We need to adapt the right half of \fig{newhex} as well, since the inner triarc has now grown larger. For this, note that each side of the inner triarc has now length $2N+1$, and so in order to use the method of the right half of \fig{newhex} the three peripheral triarcs must have a base of length $2N+2$ (in addition to having their other two sides of equal length). Since we chose $N$ to be odd, it turns out that $2N+2$ is a multiple of four, and so we can construct the required peripheral triarcs by glueing several $(2,2,4)$-triarcs together using \fig{glueing57} into a $(N+1,N+1,2N+2)$-triarc.

Moreover, the crosscap gadget shown in \fig{crosscap} can be generalized
so that  the inner 6-gon is replaced by a $6N$-gon that is surrounded by $3N$  heptagons and $3N$ pentagons, arranged alternatingly around the $6N$-gon (here it is also important that we chose $N$ to be odd).

When the time comes to insert crosscaps or glue pairs of such $6N$-gons together (after a half-twist),
we obtain a similar configuration as in \fig{handle}, but with $3N$ thick
edges. Some of these thick edges are surrounded by faces of sizes $8,6,8,6$
(as in \fig{handle}), while others are surrounded by four hexagons or by one octagon and three hexagons. Note, however, that for parity reasons we can make sure that every octagon is incident with a thick edge, and still every fourth edge on the dashed cycle is thick. Finally, the contract-uncontract operation of \fig{2x2tile} turns these faces into pentagons and heptagons only. 

It is easy to see that these changes did not hurt 3-connectivity.
Let us now argue that the resulting map $G$ has edge-width at least $z$.
Recall that the surface $S$ is obtained from a plane graph $G'$, embedded in the sphere, that is composed
of large basic triarcs $T_1,\dots,T_m$, some large parallelograms %(with sides of lengths at least $3z$) $P_1,\dots,P_q$ 
used to glue the basic triarcs together into a large triarc $T$, and a remainder $X$ comprising the material we used to enlarge $T$ into $T'$, the ring of \fig{cyl57}, and the triarc $R$. %Let  $B_1,\dots,B_m$ be the boundaries of the basic triarcs, and 
Let $L_i$ be the nucleus of $T_i$. Then $S$ is obtained
from $G'$ by glueing the crosscap gadget into some of the $6N$-gons $L_i$, and/or 
by identifying some pairs $L_i,L_j$ of the $6N$-gons to create handles. 

We claim that for every basic triarc $T_i$ such that the nucleus $L_i$ of $T_i$ is a $6N$-gon, and
\labtequ{link}{for every side $P$ of $T_i$, there is a set of $z$ pairwise disjoint $\pths{L_i}{P}$.} Indeed, recall that in order to construct $T_i$, we first surrounded $L_i$ by several pentagons and heptagons, $6N$ in total, to obtain a triarc $T_i^1$, then we performed the operation of the  right half of \fig{newhex} to obtain a triarc $T_i^2$, and finally we enlarged this into a larger triarc $T_i^3=T_i$ using the operation of \fig{glueing57} several times (this final step was described later, in the part of the current proof concerning large face-width). Now given any side $P'$ of $T_i^2$ it is possible to find, within $T_i^2$, a set of $z$ pairwise disjoint $\pths{L_i}{P'}$, see \fig{Ebelink}.
\showFigR{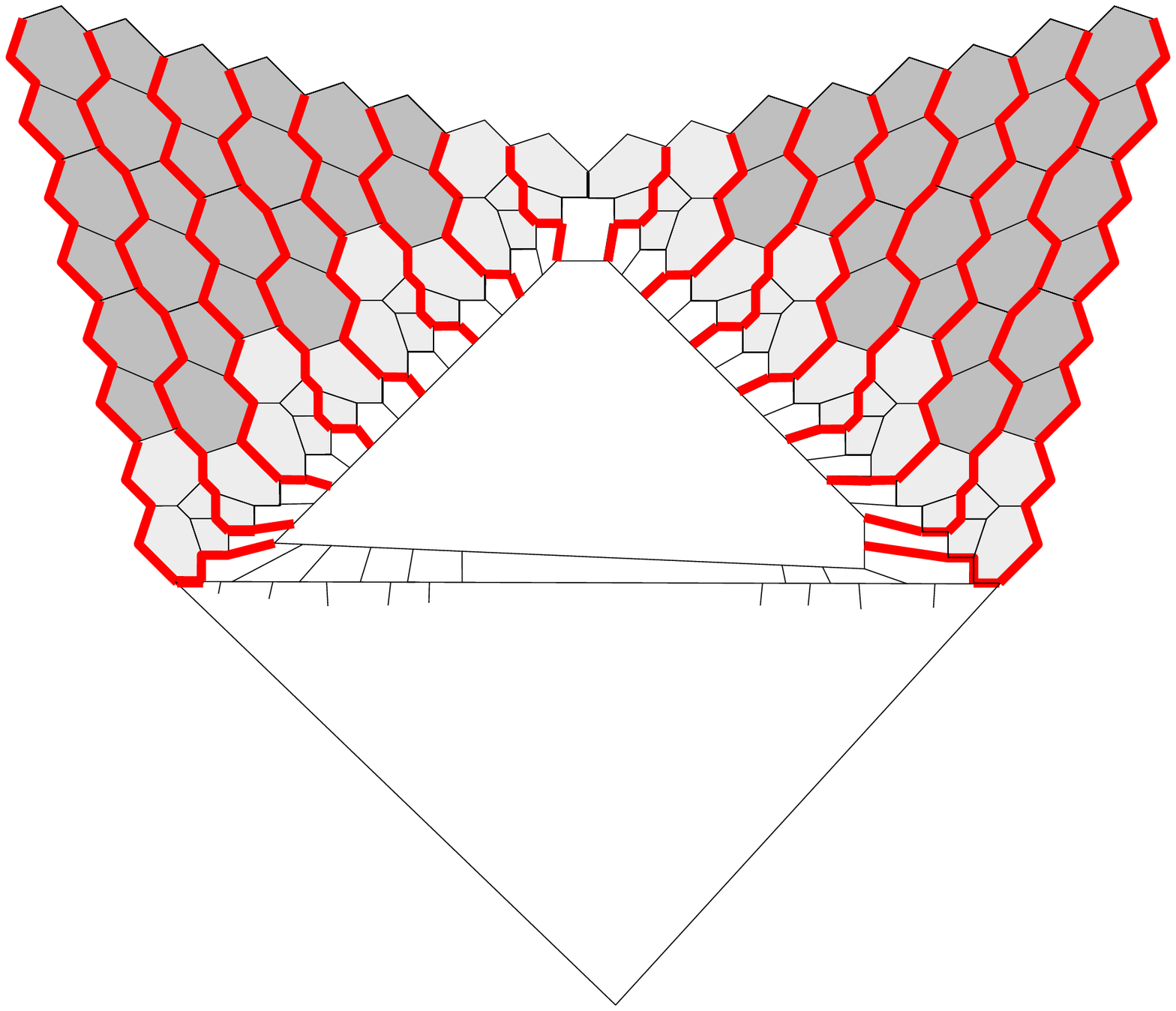}{Constructing $z$  disjoint $\pths{L_i}{P'}$, in the case that the auxilliary hexagons are replaced with
   42-gons ($6N$ for $N=7$). In light gray are
  the $(2,2,4)$-triarcs, in dark gray the \htile{}s resulting from the
  glueing operation of Figures~\ref{glueing57} and~\ref{2x2tile}. The
  empty triangle at the  bottom is a part of the triarc, isomorphic to the top-left and top-right ones. 
  As our paths do not use the bottom part, we don't show the details in the figure. The white
  triangular shape in the middle represents the nucleus.
  The 16 thick paths are the ones we need in order to prove that our graphs have large
  face-width.}{7.5cm}
Then, every time we use the operation of \fig{glueing57} while enlarging $T_i^2$ into $T_i^3$, it is possible to recursively propagate those paths to reach the side of $T_i^3$ corresponding to $P'$; if $P'$ is included within a side of $T_i^3$ then nothing needs to be done, and if not then we can propagate our paths through the parallelogram of \fig{glueing57} while keeping them disjoint (this is true even after performing the contract-uncontract operations of \fig{2x2tile}). This proves our claim \eqref{link}.

Next, we claim that any two nuclei $L_i, L_j$ can be joined by $z$ pairwise disjoint paths in $G'$. Indeed, this follows easily from \eqref{link} and the fact that whenever we glue two triarcs $T,T'$ together as in \fig{glueing57} by a parallelogram $R$ with side-lengths $m,n$, then we can find a set of $m$ pairwise disjoint paths within $R$ joining its two opposite sides of length $m$, as well as a set of $\min(m,n)$ pairwise disjoint paths within $R$ joining the sides of $T$ and $T'$ incident with $R$. 

%The rest of this section was changed after submitting
We distinguish two cases. 
\paragraph{Case 1: the surface $S$ is orientable}  
There are three types of non-contractible cycles in $G$.
The first one comes from a path $P$ in $G'$ that connects two nuclei $L_i$, $L_j$ such that these
nuclei are glued to create a handle. As the distance from each nucleus to the boundary of the 
corresponding triarc $T_i$ is at least $z$, the length of $P$ is at least $z$ as well (even at least $2z$). 

The second type of non-contractible cycle $C$ comes from a cycle $C'$ in $G'$ such that 
$|C'| \le |C|$ and $C'$ separates some nucleus $L_i$ from some other nucleus $L_j$ in $G'$.
We use the above construction of $z$ pairwise disjoint paths from $L_i$ to $L_j$ to conclude that $|C|\geq z$ as 
desired (in fact, we have $|C|\geq 2z$ because the graph is cubic and so any two paths that have a common inner vertex must have a common edge).

The last type is similar to the second one: it is a cycle $C$  that crosses some cycle $L$ of $G$ obtained by glueing two nuclei  $L_i$, $L_j$ to introduce a handle. Such a cycle $C$ comes from a $k$-tuple of paths in $G'$, where $k$ is the number of times that $C$ crosses $L$, half of which paths have ends on $L_i$, and the other half of them on $L_j$.
We may assume that none of these paths $P$ leaves the triarc containing the endpoints of $P$, for otherwise $|P|\geq z$ holds.
We will consider again the $z$ pairwise disjoint paths connecting $L_i$ to $L_j$. 
In fact, we only need to consider their parts that are contained in the triarcs $T_i$, $T_j$: 
Let these parts be $P_{i,1}$, \dots, $P_{i,z}$ (connecting $L_i$ to the boundary of $T_i$)
and $P_{j,1}$, \dots, $P_{j,z}$ (connecting $L_j$ to the boundary of $T_j$).
As the paths $P_{i,t}$ start regularly along two thirds of the nucleus $L_i$
and the same holds for $L_j$, we can use them to create $z/2$ pairwise disjoint paths $Q_1$, \dots, $Q_{z/2}$ in $G$ connecting
the boundary of $T_i$ to the boundary of $T_j$. Note that the part of the surface $S$ containing $T_i$
and $T_j$ is a cylinder, and the cycle $C$ goes around this cylinder. Thus $C$   must intersect all of the paths $Q_t$. As the graph is cubic, each intersection with a path has to use at least two vertices, proving again that $|C| \ge z$. 

\paragraph{Case 2: $S$ is non-orientable}  
In this case a non-contractible cycle $C$ in $G$ will either yield a cycle $C'$ as above, in which case the same
argument applies, or it will yield a path $P'$ in $G'$ whose endpoints were identified when introducing crosscaps. Recall that we made every basic
triarc used in the construction large enough that the distance from its nucleus to the boundary of the triarc is at least $z$, thus $P'$ is, \obda,
contained within one of the triarcs in which a crosscap was introduced. With the help of \fig{Ebelink} and \fig{crosscap} (modified with a $6N$-gon
replacing the hexagon as described above) it is now not hard to see  that $|P'|\geq z$ as desired.

\end{proof}

\section{Other neutral sequences} \label{other}

In this paper we concentrated on the neutral sequence $(0,0,1,0,1)$, but we believe that our methods and results apply in a much more general setting ---see also \Sr{outl}--- and it is the purpose of this section to explain this.
%We would be tempted to conjecture that an Eberhard-like theorem holds for any
%neutral sequence, not just for $(0,0,1,0,1)$. However at this point we believe that some additional global 
%constraints may also apply. Nevertheless, some parts of our construction can be extended 

In \Sr{main} we showed that every plausible sequence can be extended into a realizable one by adding pentagons and heptagons only. In what follows we are going to give a rough sketch of a proof that an arbitrary neutral sequence $s$ can be used to extend any plausible sequence into a realizable one under the assumption that a couple of basic building blocks can be constructed using precisely the faces that appear in some multiple of $s$. We expect that our construction will help yield more general results in the future, by showing that these building blocks can indeed be constructed.

%\begin{theorem}\label{}
So let $p=(p_3,p_4,\ldots,p_r)$ be a plausible sequence for the sphere or the torus, and let $s=(p'_3,p'_4,\ldots,p'_t)$ be a neutral sequence. In order to prove that there is some $n$ so that $p+ns$ is realizable, it suffices to find some $k\in \N$ for which it is possible to construct the following building blocks using precisely the faces that appear in some multiple of $s$:
\begin{enumerate}
\item \label{ii} a $(k,k,k)$-triarc;
\item \label{iii} a $(k,k,k-1)$-triarc;
\item \label{i} for every non-zero entry $p_l$ in $p$, a triarc containing a face of size $l$, such that the length of two of the sides of this triarc is
a multiple of $k$;
\item \label{iv} a ``ring'' like the one in \fig{cyl57} (using the faces from $s$ in the right proportion rather than pentagons and heptagons) for combining two
equilateral triarcs.
\end{enumerate}
%Note that if $s$ has more than 2 non-zero entries then 
%\end{theorem} 
% *** ---- *** 
%\begin{proof}

\showFigR{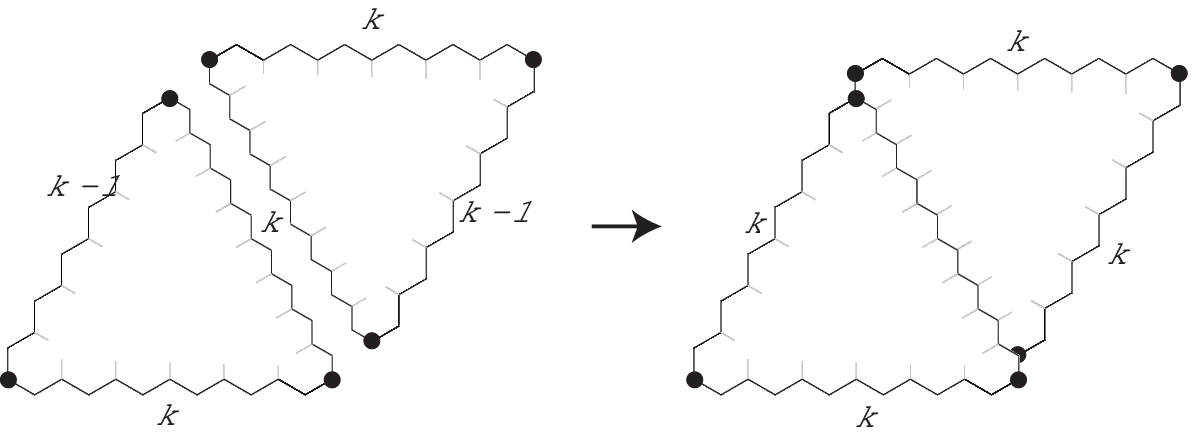}{Constructing a parallelogram out 
of two $(k,k,k-1)$-triarcs.}{10.5cm}

Indeed, to begin with, construct a parallelogram with all sides of length $k$ out of two $(k,k,k-1)$-triarcs (supplied by~\ref{iii}) as shown in \fig{2triarcs}. 
(In figures explaining our construction, we shall use triarcs made of hexagonal
faces, but this is for illustration purposes only; in fact they have to be made of multiples of $s$.)
This also allows us to construct any parallelogram with dimensions $mk,lk$ for every $m,l\in \N$.

Next, similarly to the construction in \Tr{ebe57}, construct a `basic' triarc as in \ref{i} for each face-size $l$ for which $p_l\neq 0$; in fact, we construct $p_l$ copies of this basic triarc \fe\ $l$. Then, using the parallelograms we constructed earlier, we recursively glue all those triarcs together into a single triarc $T$, in a manner very similar to the operation of \fig{glueing57}.

%\showFig{glueing2}{}

%\showFig{decreasing}{}

%\showFigR{increasing}{Increasing the length of each side of a triarc by $k$.}{10.5cm}

By recursively glueing the resulting triarc with a $(k,k,k-1)$-triarc provided 
by \ref{iii} using the glueing operation of \fig{glueing57}, we can transform $T$ into an equilateral $(mk,mk,mk)$-triarc $T'$  for some (large) $m\in \N$.

Using the glueing operation of \fig{glueing57} it is possible to construct a triarc $R$ with the same side-lengths as $T'$, using only $(k,k,k)$-triarcs  (provided by \ref{ii})  and the above parallelograms.%; see \fig{increasing}.

In the case of the sphere, we can combine $R$ and $T'$ by using the ``ring''
provided by \ref{iv} to complete the construction.

\showFigR{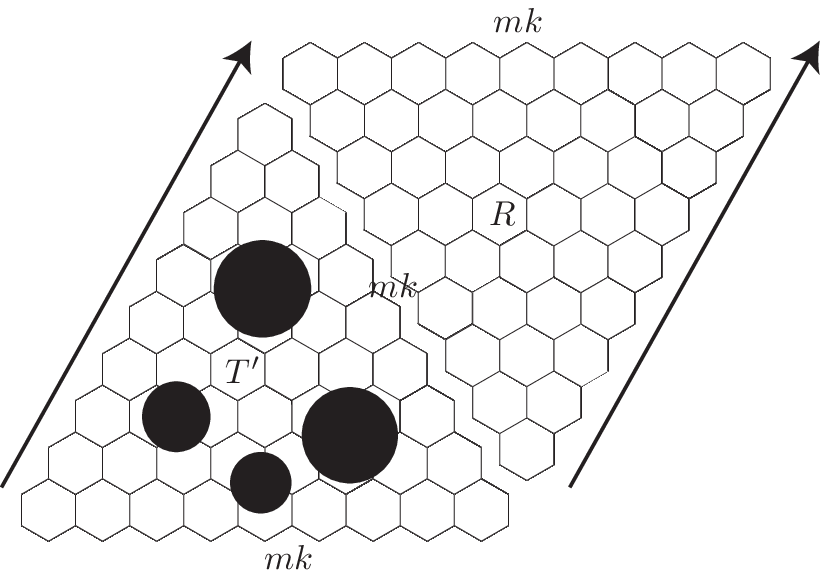}{Glueing $R$ and $T'$ together. The black dots depict the faces imposed by the sequence $p$.}{6.4cm}

If the underlying surface $S$ is the torus, we
glue $R$ and $T'$ together along one of their sides to obtain a parallelogram, and glue two opposite sides of this parallelogram together to obtain a cylinder $C$ both of whose bounding cycles are in-out alternating cycles of length $mk$, see \fig{cylinder}. 
We then glue the two bounding cycles of $C$ together to obtain a realization of a torus. 

If $p$ is plausible for some other surface $S$, then we would need additional gadgets like those used in the proof of Corollary~\ref{cor}.

%For the sphere we have to do some more work. One possibility is to reserve some of the faces in $p$ of total curvature \sss 12 which we do not put into $T'$, and try to use those faces (and the faces in $s$) to build two discs in order to close the cylinder $C$ into a sphere. Another possibility is to try to construct a ``ring'' as in~\ref{iv} and imitate our construction in \Tr{ebe57}.

\section{Outlook} \label{outl}

Trying to achieve a better understanding of the implications of Euler's formula, we studied the question of whether, given a plausible sequence $p$, and a neutral sequence $q$, it is possible to combine $p$ and $q$ into a realizable sequence $p+nq$, but we did so in very restricted cases. The general problem remains wide open; in particular, we would be interested to see an answer to the following problem.

\begin{problem}
 Given a closed surface $S$, is it true that for every plausible sequence $p$ for $S$, and every neutral sequence $q$, there is an $\nin$ such that $p+nq$ is realizable in $S$ with the exception of only finitely many pairs $(p,q)$? 
\end{problem}
(As mentioned in the introduction, if $S$ is the torus then the list of exceptional pairs $(p,q)$  is not empty.)

\end{document}